\documentclass[a4paper,reqno,12pt]{amsart}
\usepackage[utf8]{inputenc}
\usepackage[T1,T2A]{fontenc}
\usepackage[english]{babel}
\usepackage{amsmath}
\usepackage{amssymb}
\usepackage{amscd}
\usepackage{amsthm}
\usepackage{euscript}
\usepackage[all]{xy}
\usepackage[format=hang]{caption}
\usepackage{comment}
\usepackage{tikz}
\usetikzlibrary{patterns}
\newcommand{\drawcross}[2]{%
  \draw[gray, line width=1pt] ([shift={(-0.13,-0.13)}]#1,#2) -- ++(0.26,0.26);
  \draw[gray, line width=1pt] ([shift={(-0.13,0.13)}]#1,#2) -- ++(0.26,-0.26);
}

\newcommand{\drawblackcross}[2]{%
  \draw[black, line width=1pt] ([shift={(-0.13,-0.13)}]#1,#2) -- ++(0.26,0.26);
  \draw[black, line width=1pt] ([shift={(-0.13,0.13)}]#1,#2) -- ++(0.26,-0.26);
}

\newtheorem{proposition}{Proposition}

\newtheorem{lemma}{Lemma}
\newtheorem{theorem}{Theorem}

\newtheorem{corollary}{Corollary}
\theoremstyle{definition}
\newtheorem{definition}{Definition}
\newtheorem{example}{Example}
\theoremstyle{remark}
\newtheorem {remark}{Remark}

\DeclareMathOperator{\Spec}{Spec}
\DeclareMathOperator{\Hom}{Hom}
\DeclareMathOperator{\Aut}{Aut}

\DeclareMathOperator{\GL}{GL}

\DeclareMathOperator{\codim}{codim}

\def\GG{{\mathbb G}}

\def\ZZ{{\mathbb Z}}

\def\NN{{\mathbb N}}

\def\QQ{{\mathbb Q}}

\def\AA{{\mathbb A}}

\renewcommand{\phi}{\varphi}
\renewcommand{\ge}{\geqslant}
\renewcommand{\le}{\leqslant}

\newcommand{\ti}{\widetilde}
\def\bangle#1{{\langle #1 \rangle}}
\begin{document}
\date{}
\title{Automorphism groups of non-normal affine toric surfaces}
\author{Kirill Selin}
\address{Faculty of Computer Science, HSE University, Pokrovsky Boulevard 11, Moscow, 109028 Russia}
\email{k.selin@hse.ru}


\thanks{This article is an output of a research project (HSE-BR-2025-22) implemented as part of the
Basic Research Program at HSE University.}

\begin{abstract}
We investigate the automorphism group of a non-normal affine toric surface. It is known that for a normal affine toric surface the connected component of the automorphism group is generated by the acting torus and root subgroups. While this fact is also true for some classes of non-normal affine toric surfaces, it fails in general. We give an example of a non-normal affine toric surface, whose singular locus is a point and the connected component of the automorphism group is not generated by the acting torus and root subgroups.
\end{abstract}

\maketitle


\section{Introduction}
\label{sec1}

The study of automorphism groups of algebraic varieties is a classical subject in algebraic geometry. For normal complete toric varieties, these groups can be described in terms of the combinatorial data of the underlying lattices and fans \cite{De-1, Oda-1, Cox}. In the case of the affine plane, the classical Jung-van der Kulk theorem claims that the automorphism group is an amalgam of the group of affine transformations and the de Jonquières group. It is shown by Arzhantsev and Zaidenberg \cite{AZ13} that for any non-degenerate normal affine toric surface the automorphism group is an amalgam. These results imply that the connected component of the automorphism group is generated by the acting torus $T$ and root subgroups. Moreover, root subgroups are encoded by Demazure roots; see Section \ref{Normal} for details.

In this paper, we study the automorphism group of a non-normal affine toric surface. It is well known that non-normal affine toric varieties correspond to non-saturated affine monoids \cite[Chapter 1]{CLS11}. The concept of a Demazure root in this settings is introduced by  D\'{\i}az and Liendo in \cite{DL23}. While the automorphism groups of certain non-normal toric varieties have been addressed in the literature (e.g., \cite{BG16}, \cite{DL23}), the general case remains poorly understood.

Recall that any non-degenerate affine toric surface consists of four $T$-orbits: a dense open orbit, two one-dimensional orbits, and a fixed point. The singular locus of a non-normal surface is a union of non-open orbits. The type of this locus dictates the existence of Demazure roots associated with the corresponding rays. Based on this, we divide the surfaces into three cases.

\begin{enumerate}
    \item \emph{Both one-dimensional torus orbits lie in the singular locus.} The surface admits no Demazure roots, and the automorphism group is either isomorphic to the acting torus $T$ or the semidirect product $\ZZ / 2\ZZ \ltimes T$.
    \item \emph{Only one one-dimensional torus orbit lies in the singular locus.} All Demazure roots are associated with a single ray, which implies that the corresponding root subgroups commute. The automorphism group is generated by the acting torus and root subgroups, forming a semidirect product $\mathcal{U} \rtimes T$, where $\mathcal{U}$ the additive group of a countably dimensional vector space.
    \item \emph{The singular locus is a fixed point.} Demazure roots are associated with both rays of the cone. Unlike the previous cases and the normal setting, we show that the connected component of the automorphism group is not necessarily generated by the torus and root subgroups. Specifically, we provide a counterexample by showing that the automorphism group of the surface $X \subset \AA^4$ given by ${z^3 = w^2}, {zx^2 = y^2}, {xw = yz}, {yw = xz^2}$ contains an irreducible algebraic family of automorphisms that cannot be factored as a product of torus elements and root subgroups on $X$ (Theorem \ref{thmCase3}). This shows that the third case deserves further investigation.
\end{enumerate}

Surfaces from the first two cases are rigid and almost rigid correspondingly. The automorphism group of rigid and almost rigid affine toric varieties is described in \cite{BG16}.

The author expresses his sincere gratitude to Ivan Arzhantsev for posing the problem, constant attention to this work, and valuable suggestions.

\section{Automorphisms of normal toric surfaces}
\label{Normal}

In this section, we review known results on the automorphism groups of normal affine toric surfaces, primarily following \cite{AZ13}. We begin by recalling the combinatorial description of affine toric varieties based on \cite[Sections 1.3, 2.2]{Ful93}. The ground field $k$ is supposed to be algebraically closed and of characteristic zero.

Let $N \cong \ZZ^n$ be a lattice of rank $n$ and $M = \Hom(N, \ZZ)$ be the dual lattice. Denote by $\langle m, u\rangle$ the pairing of $m \in M$ and $u \in N$.
Consider $N_{\QQ} = N \otimes_{\ZZ} \QQ$ and $M_{\QQ} = M \otimes_{\ZZ} \QQ$.
Let $\sigma \subset N_{\QQ}$ be a strongly convex rational polyhedral cone.
The dual cone is defined as
$$\sigma^{\vee} = \{u \in M_{\QQ} : \langle u, v \rangle \geq 0 \text{ for all } v \in \sigma\},$$
Since the cone $\sigma$ is rational and polyhedral, the monoid $S_{\sigma} = \sigma^{\vee} \cap M$ is finitely generated. It defines a finitely generated commutative $k$-algebra $k[S_\sigma]$. As a vector space, $k[S_\sigma]$ is spanned by characters $\chi^m$ for $m \in S_\sigma$, with multiplication $\chi^u \cdot \chi^v = \chi^{u+v}$. The affine toric variety $U_{\sigma}$ associated with the cone $\sigma$ is defined as $U_{\sigma} = \Spec(k[S_{\sigma}])$.

An alternative approach to describing affine toric varieties relies on quotient constructions.
Let $N' \subset N$ be a sublattice of finite index $d = [N:N']$, and $M \subset M'$ be the dual lattices. Fix a primitive $d$-th root of unity $\zeta \in k^\times$ and consider an action of abelian group $G = N/N'$ on the coordinate ring $k[M']$ via
\[
\bar{v} \cdot \chi^{u'} = \zeta^{d \langle u', v \rangle} \cdot \chi^{u'},
\]
for all $\bar{v} \in N/N'$ and $u' \in M'$.
The corresponding ring of invariants is ${k[M']^G = k[M]}$.

Suppose $\sigma$ is a non-regular simplicial $n$-dimensional cone in $N$.
Let $N' \subset N$ be the sublattice generated by the primitive vectors on rays of the cone $\sigma$.
This gives a regular cone $\sigma'$ in $N'$ and the quotient morphism $\AA^n \cong U_{\sigma'} \to U_{\sigma}$ with respect to the action of the abelian group $G = N/N'$. In particular,
\[
U_{\sigma} = U_{\sigma'}/G \cong \AA^n/G.
\]

For normal singular affine toric surfaces, the lattice $N$ has rank $2$, so we can choose some basis $f_1, f_2$ in $N$. We can assume that $\sigma$ is generated by $f_2$ and $ef_1 + df_2$ for some integers $d, e$ satisfying $d \geq 2$, $0 < e < d$, and $\gcd(d,e) = 1$; see \cite[Chapter 2]{Ful93}.
The associated monoid algebra takes the form
\[
A_\sigma = k[S_\sigma] = \bigoplus_{\substack{i \geq 0 \\ ei + dj \geq 0}} k \cdot x^iy^j.
\]
By mapping $x$ to $X^eY$ and $y$ to $X^d$ we obtain the inclusion
\[
A_\sigma \subset k[X,Y].
\]

Let $G$ be the group of $d$-th roots of unity in $k$.
It acts on $\AA^2$ via $\zeta \cdot (u, v) = (\zeta^{-1} u, \zeta^e v)$. Consequently, $G$ acts on the coordinate ring $k[X,Y]$ by ${\zeta \cdot F(X, Y) = F(\zeta X, \zeta^{-e} Y)}$. The invariant subalgebra is
\[
A_\sigma = k[X,Y]^G,
\]
making $U_\sigma$ the quotient $\AA^2/G$. We denote the image of $G$ in $\Aut(\AA^2)$ by $G_{d,e}$. The corresponing affine toric surface is denoted by $X_{d, e}$.

Now we come to the automorphism groups. Let us begin with degenerated normal affine toric surfaces.

Recall that a toric variety $X$ is degenerate if $k[X]^\times \neq k^\times$. Equivalently, this means that $X \cong T_1\times X_1$, where $T_1$ is a non-trivial subtorus and $X_1$ is a toric variety. The only degenerate normal affine toric surfaces are $T^2$ and $T^1\times\AA^1$. The automorphism group of the former one consists of automorphisms of the form
$$(t, x) \mapsto (\alpha t^{\pm1}, \beta t^ax + f(t)),$$
where $\alpha, \beta \in k^\times$, $a \in \ZZ$ and $f \in k[t, t^{-1}]$. The automorphism group of the two-dimensional torus is $\Aut(T^2) = T^2 \rtimes \GL_2(\ZZ)$.

By $\mathrm{JONQ}^+(\AA^2)$ ($\mathrm{JONQ}^-(\AA^2)$, respectively) we denote the group of de Jonquières transformations
\[
\Phi^+: (x, y) \mapsto (\alpha x + f(y), \beta y + \gamma),
\]
respectively,
\[
\Phi^-: (x, y) \mapsto (\alpha x + \gamma, \beta y + f(x)),
\]
where $\alpha, \beta \in k^\times$, $\gamma \in k$, and $f \in k[t]$. The group of affine transformations of $\AA^2$ is denoted by $\mathrm{Aff}(\AA^2)$ and the intersection is $\mathrm{Aff}^\pm(\AA^2) = \mathrm{JONQ}^\pm(\AA^2) \cap \mathrm{Aff}(\AA^2)$.

The automorphism group $\Aut(\AA^2)$ is described by the following Jung-van der Kulk theorem
\begin{theorem}[\cite{Jung42, Kulk53, Kambayashi75}]
    The automorphism group $\Aut(\AA^2)$ is the amalgam
    $$\Aut(\AA^2) = \mathrm{JONQ}^+(\AA^2)*_{\mathrm{Aff}^+(\AA^2)}\mathrm{Aff}(\AA^2).$$
\end{theorem}

Following \cite{AZ13}, we briefly recall the structure of the automorphism group for a normal affine toric surface $X_{d, e} = \AA^2 / G_{d, e}$.

Consider the subgroups ${\mathrm{Jonq}^\pm(\AA^2)\subseteq\mathrm{JONQ}^\pm(\AA^2)}$ consisting of the transformations
$$\phi^+: (x, y) \mapsto (\alpha x + f(y), \beta y),$$
and
$$\phi^-: (x, y) \mapsto (\alpha x, \beta y + f(x)),$$\
where $\alpha, \beta \in k^\times$ and $f \in k[t]$.

We introduce the following notation:
\begin{itemize}
    \item $N_{d,e}$ is the normalizer of $G_{d,e}$ in $\GL(2,k)$;
    \item $\mathcal{N}_{d,e}$ is the normalizer of $G_{d,e}$ in $\Aut(\AA^2)$.
\end{itemize}

Define the respective subgroups
\[
N_{d,e}^{\pm} = N_{d,e} \cap \mathrm{Jonq}^{\pm}(\AA^2), \qquad
\mathcal{N}_{d,e}^{\pm} = \mathcal{N}_{d,e} \cap \mathrm{Jonq}^{\pm}(\AA^2).
\]

Recall that the polynomial ring $A = k[t]$ possesses a $\ZZ/d\ZZ$-grading
\[
    A = \bigoplus_{i=0}^{d-1}A_{d, i}, \; \text{where} \; A_{d, i} = t^ik[t^d].
\]

\begin{lemma}[{\cite[Lemma 4.5]{AZ13}}]\label{AZLem45}
The groups $\mathcal{N}_{d,e}^+$ and $\mathcal{N}_{d,e}^-$ consist of all de Jonquières transformations taking the form
\[
(x,y) \mapsto (\alpha x, \gamma y + f(x)) \quad \text{and} \quad (x,y) \mapsto (\alpha x + g(y), \gamma y),
\]
respectively, where $f(x) \in A_{d,e}$, $g(y) \in A_{d,e'}$, and $ee' \equiv 1 \mod d$. Furthermore, $\mathcal{N}_{d,e}^\pm$ is the centralizer of $G_{d,e}$ in $\mathrm{Jonq}^\pm(\AA^2)$.
\end{lemma}

\begin{theorem}[{\cite[Theorem 4.2]{AZ13}}]\label{AZTh42}
    If $e^2 \not\equiv 1 \pmod{d}$, then
    \[
        \Aut(X_{d,e}) \simeq \mathcal{N}_{d,e}^+ / G_{d,e} *_{T / G_{d,e}} \mathcal{N}_{d,e}^- / G_{d,e},
    \]
    while for $e^2 \equiv 1 \pmod{d}$ we have
    \[
    \Aut(X_{d,e}) \simeq \mathcal{N}_{d,e}^+ / G_{d,e} *_{N_{d,e}^+ / G_{d,e}} N_{d, e} / G_{d,e}.
    \]
\end{theorem}

\begin{definition}
    An element $e \in M$ is called a Demazure root of a rational polyhedral cone $\sigma \subset N_\QQ$ associated with a ray $\rho \subset \sigma$ if $\langle e, u_\rho \rangle = - 1$ for the primitive generator $u_\rho \in N$ of the ray $\rho$, and $\langle e, u_{\rho'}\rangle > 0$ for all the rays $\rho' \subset \sigma$ different from $\rho$.
\end{definition}

Denote by $\mathcal{R}_\rho(\sigma)$ the set of all Demazure roots of $\sigma$ associated with a ray $\rho$, and ${\mathcal{R}(\sigma) = \coprod_\rho \mathcal{R}_\rho(\sigma)}$. Let $\sigma$ be a cone of an affine toric surface $X$. Every Demazure root $e \in \mathcal{R}_\rho(\sigma)$ induces a derivation $\delta_e$ of the coordinate algebra $k[X]$ via
\[
\delta_e(\chi^m) = \bangle{m, u_\rho}\chi^{m + e}
\]
for any $m \in S$.
This derivation is locally nilpotent, i.e, for any element $f \in k[X]$, there exists an integer $n \in \NN$ with $\delta_e^n f = 0$. This gives a family of automorphisms of $k[X]$ of the form $\exp(\alpha \delta_e)$, where $\alpha \in k$.
Thus, it defines a $\GG_a$-action on $X$. The corresponding subgroup of $\Aut(X)$ is normalized by $T$.

A unipotent one-dimensional subgroup of $\Aut(X)$ normalized by $T$ is called a root subgroup. Any root subgroup of an affine toric variety can be defined using some Demazure root.

\begin{corollary}
The connected component $\Aut(X_{d,e})^0$ of the automorphism group of a toric surface $X_{d,e}$ is generated by the acting torus and root subgroups.
\end{corollary}

\begin{proof}
By Theorem \ref{AZTh42}, $\Aut(X_{d,e})$ is an amalgam of the groups $\mathcal{N}_{d,e}^+ / G_{d,e}, \,\, \mathcal{N}_{d,e}^- / G_{d,e}$ or $\mathcal{N}_{d,e}^+ / G_{d,e}, \,\, N_{d,e} / G_{d,e}$.
By Lemma \ref{AZLem45}, connected components of each of these groups consist of de Jonquières transformations, which factor as a product of torus elements and triangular transformations mapping ${(x,y) \mapsto (x, y + f(x))}$ or $(x,y) \mapsto (x + g(y), y)$.
These triangular transformations are products of root automorphisms of $X_{d,e}$.
\end{proof}

For a broader context, automorphism groups of arbitrary normal affine surfaces are studied in \cite{Kovalenko2018}. The authors show that the amalgamated structure of the automorphism group is common for Gizatullin surfaces.

\section{Affine monoids and non-normal surfaces}

\subsection{Toric varieties and affine monoids}
A commutative finitely generated monoid $S$ is called affine if it can be embedded into a free abelian group of finite rank. Every affine monoid defines an irreducible affine variety $X_S = \operatorname{Spec} k[S]$. Throughout this paper, we assume all monoids to be affine.

Let $M \cong \mathbb{Z}^n$ denote the lattice generated by an affine monoid $S$. The saturation $S^{\operatorname{sat}}$ of $S$ is the set of all $m \in M$ with $lm \in S$ for some positive integer $l$. We say that $S$ is saturated if $S = S^{\operatorname{sat}}$. It is a well-known fact that the variety $X_S$ is normal if and only if $S$ is saturated \cite[Theorem~1.3.5]{CLS11}.

There is an action of the torus $T = N \otimes_{\mathbb{Z}} \mathbb{G}_m \cong \mathbb{G}_m^n$ on $k[S]$, where $N = \operatorname{Hom}(M, \mathbb{Z})$. This action is given by
\[
(u \otimes t) \cdot \chi^m = t^{\langle m, u \rangle}\chi^m.
\]

This induces a torus action on $X_S$ with a dense open orbit, because the coordinate ring $k[X_S] = k[S]$ embeds into $k[T] = k[M]$ as a finitely generated subalgebra. Thus, $X_S$ is a toric variety. Conversely, every affine toric variety $X$ is isomorphic to $X_{S(X)}$, where $S(X)$ is the monoid of $T$-characters that are regular on $X$.

Any saturated monoid $S$ is of the form $S = M \cap \sigma(S)^\vee$, where $\sigma(S)$ is a rational polyhedral cone defined by
$$\sigma(S) = \{u \in N_{\mathbb{Q}} : \langle m, u \rangle \ge 0 \ \forall m \in S\}$$

A non-saturated monoid $S$ is a submonoid of its saturation $S^\mathrm{sat}$. Thus, it can be described with the cone $\sigma(S)$ and the set of ''holes'' $A(S) = S^\mathrm{sat} \setminus S$; see \cite{BG16}. Note that for any affine monoid $S$ we have $\sigma(S) = \sigma(S^\mathrm{sat})$.
If $X$ is an affine toric variety, we denote $\sigma(X) = \sigma(S(X))$.

\subsection{Almost saturated faces} Here we review the notion of almost saturated faces of a cone introduced in \cite{TY08}, see also \cite{BG16}.
\begin{definition}
    An element $m \in S$ is called a saturation point of a monoid $S$ if $m + S^{sat} \subset S$.
A face $\tau$ of the dual cone $\sigma(S)^\vee$ is called almost saturated if it contains a saturation point $m \in \tau \cap M$; otherwise, it is called nowhere saturated.
\end{definition}

The following fact is addressed in \cite[Exercise 7.15]{MS05}.
\begin{lemma} \label{almostSat}
    The maximal face $\sigma(S)^\vee$ is almost saturated.
\end{lemma}

Any ray $\rho$ of a cone $\sigma(S)$ defines a face $\sigma(S)^\vee \cap \rho^\bot$ of the cone $\sigma(S)^\vee$. We can describe the property of $\sigma(S)^\vee \cap \rho^\bot$ to be nowhere saturated in terms of the hole set $A(S)$.
\begin{lemma} \label{rhoSing}
    For a ray $\rho \subseteq \sigma(S)$ the face $\sigma(S)^\vee \cap \rho^\bot$ is nowhere saturated if and only if one of the following two conditions hold
    \begin{enumerate}
     \item The intersection $A(S) \cap \rho^\bot$ is infinite;
     \item The set $A(S)$ contains the first-level line $\{m \in S^{sat}: \langle m, u_{\rho}\rangle = 1\}$ parallel to $\rho^\bot$.
    \end{enumerate}
\end{lemma}

\begin{proof}\
    Denote by $\rho'$ a ray of $\sigma(S)$ different from $\rho$. Suppose neither of the stated conditions hold.
This implies there exists a positive integer $n$ such that for any $m \in S^{sat} \cap \rho^\bot$ with $\langle m, u_{\rho'}\rangle \geq n$ we have $m \in S$, and there exists an element $m_1 \in S$ satisfying $\langle m_1, u_\rho\rangle = 1$.

By Lemma \ref{almostSat}, there is an element $m' \in S$ with $m' + S^{sat} \subseteq S$. Denote $l = \bangle{m', u_\rho}$, and for $j = 1, \dots, l - 1$ \ $m_j = jm_1$. We have $\bangle{m_j, u_\rho} = j$, so there is a character on every parallel line between the lines defined by the equations $\bangle{m, u_\rho} = 0$ and $\bangle{m, u_\rho} = \bangle{m',u_\rho}$.

Define $K = \max\left(n + \bangle{m_{l-1}, u_{\rho'}}, \bangle{m', u_{\rho'}}\right)$.
For every $m \in S^{sat}$ with $\langle m, u_{\rho'}\rangle \geq K$, we have $m \in S$, since either $m \in m' + S^\mathrm{sat}$ or $\bangle{m, u_\rho} = j \le l - 1$ and $m$ is far enough from $m_j$ along $\rho^\bot$.

Thus, there is an element $m \in \rho^\bot \cap S$ with $m + S^{sat} \subset S$, proving that the ray $\rho^\bot \cap \sigma(S)^\vee$ is almost saturated.

Conversely, if either of the given conditions hold, it follows that for any $m \in S \cap \rho^\bot$, the set $m + S^{sat}$ does not lie in $S$. This means that $\sigma(S)^\vee \cap \rho^\bot$ is nowhere saturated.
\end{proof}

For a given ray $\rho$ of $\sigma(S)$, a monoid $m\ZZ + S$ is independent of the choice of $m \in \rho^\bot \cap S$.  We denote this monoid by $S_\rho$.
\begin{lemma} \label{rhoSingAlmostSat}
    The ray $\rho^\bot \cap \sigma(S)^\vee$ is almost saturated if and only if the corresponding monoid $S_\rho$ is saturated.
\end{lemma}
\begin{proof}
    If $\rho^\bot \cap \sigma(S)^\vee$ is almost saturated, there exists an $m \in \rho^\bot \cap S$ with $m + S^{sat} \subset S$. Thus, the extended monoid $S_\rho = \ZZ m + S$ contains the set $\ZZ m + S^{sat} = (S_\rho)_{sat}$, proving that $S_\rho$ is saturated.

    On the other hand, if the ray $\rho^\bot \cap \sigma(S)^\vee$ is nowhere saturated, by Lemma \ref{rhoSing} either $\rho^\bot \cap A(S)$ is infinite or the first-level line $\{m \in S^{sat} : \langle m, u_\rho\rangle = 1\}$ lies in $A(S)$. In both cases $S_\rho$ fails to be saturated.
\end{proof}

\subsection{Torus orbits}

Since the singular locus of a toric variety $X$ is torus invariant, it is the union of torus orbits. These orbits can be described combinatorially in terms of the faces of the cone $\sigma(X)$; see \cite{CLS11}.
\begin{theorem} \cite[Theorem 3.2.6]{CLS11}
    There is a one-to-one correspondence between torus orbits of a toric variety $X$ and faces of its cone $\sigma(X)$.
\end{theorem}
For a face $\tau \subseteq \sigma(X)$ we denote by $O(\tau)$ the corresponding torus orbit. We have $\dim O(\tau) = \codim\tau$.

For an affine toric surface associated with a non-degenerate cone $\sigma(X)$, there exist four orbits corresponding to four cone faces: the open orbit, two one-dimensional orbits, and a fixed point.

The union of a one-dimensional orbit and a fixed point is closed. Therefore, there are three cases for the singular locus: it is a closed torus orbit, it is the closure of a single one-dimensional torus orbit, or it is the union of the closures of both one-dimensional orbits.

Denote the two rays of the cone $\sigma(X)$ by $\rho$ and $\rho'$. We now show when the orbit closure $\overline{O(\rho)}$ corresponding to the ray $\rho$ consists of singular points.

\begin{lemma}
    The orbit closure $\overline{O(\rho)}$ consists of singular points if and only if the monoid $S(X)_{\rho'}$ is not saturated; equivalently, the ray $\sigma(X)^\vee \cap \rho'^\bot$ is nowhere saturated.
\end{lemma}

\begin{proof}
    Consider the orbit closure $\overline{O(\rho')}$. The complement $X \setminus \overline{O(\rho')}$ is a toric variety associated with the monoid $S(X)_{\rho'}$. It has two torus orbits: one open orbit and one one-dimensional orbit $O(\rho)$.

    Since the singular locus of any normal variety has a codimension of at least $2$, the subvariety $X \setminus \overline{O(\rho')}$ is either regular or non-normal. Therefore, the orbit $O(\rho)$ is singular if and only if the monoid $S(X)_{\rho'}$ is not saturated.
\end{proof}

\subsection{Demazure roots and $\GG_a$-actions}
The defenition of Demazure roots of a rational polyhedral cone $\sigma$ implies that for any element $m \in S_\sigma = \sigma^\vee \cap M$ and a root $e \in \mathcal{R}_\rho(\sigma)$ for the ray $\rho \subset \sigma$ the sum $e + m$ lies in the monoid $S_\sigma$ if and only if $\langle m, u_\rho\rangle > 0$. This is used to define Demazure roots of an arbitrary affine monoid, as detailed in \cite{DL23}.
\begin{definition}
    An element $e \in M$ is called a Demazure root of the monoid $S$ associated with a ray $\rho \subseteq \sigma(S)$ if $\langle e, u_\rho\rangle = -1$, and $e + m \in S$ for $m \in S$ with $\langle m, u_\rho\rangle > 0$.
\end{definition}

We denote the set of all Demazure roots of a monoid associated with a ray $\rho \subseteq \sigma(S)$ by $\mathcal{R}_\rho(S)$, and $\mathcal{R}(S) = \coprod_{\rho}\mathcal{R}_\rho(S)$. If a monoid $S$ is saturated, $\mathcal{R}(S)$ coincides with $\mathcal{R}(\sigma(S))$. Moreover, for a monoid $S$ we have $\mathcal{R}(S) \subseteq \mathcal{R}(S^{\operatorname{sat}}) = \mathcal{R}(\sigma(S))$. For an affine toric variety $X$ we also denote $\mathcal{R}_\rho(X) = \mathcal{R}_\rho(S(X))$ for a ray $\rho \subseteq \sigma(X)$ and $\mathcal{R}(X) = \mathcal{R}(S(X))$.

We now consider the connection between nowhere saturated rays of $\sigma(S)^\vee$ and Demazure roots.
\begin{proposition}
    For a ray $\rho \subseteq \sigma(S)$ the face $\sigma(S)^\vee \cap \rho^\bot$ is nowhere saturated if and only if $\mathcal{R}_\rho(S) = \varnothing$.
\end{proposition}
\begin{proof}
    We first prove the direct implication. By Lemma \ref{rhoSing}, if the ray $\sigma^\vee \cap \rho^\bot$ is nowhere saturated, the set of holes $A$ satisfies one of two conditions.

    Suppose first that condition (1) of Lemma \ref{rhoSing} holds, i.e. the set of holes $A$ contains the first-level line parallel to $\rho^\bot$. Let $m_0 \in S$ be the character with minimal positive value $l = \langle m_0, u_\rho\rangle$, we have $l > 1$. Consider a line $L$ defined by the equation $\langle m, u_\rho\rangle = l - 1$ for $m \in M$. It follows that $L \cap S = \varnothing$.

Thus, given a Demazure root $e \in \mathcal{R}_\rho(S^\mathrm{sat})$, we have $\langle m_0 + e, u_\rho\rangle = l - 1$. The sum $m_0 + e$ lies on the line $L$ and so does not lie in $S$. Therefore, $e$ is not a Demazure root of $S$.

     Now assume the condition (2) of Lemma \ref{rhoSing} holds, so that the intersection $\rho^\bot \cap A$ is infinite. Since the maximal face of $\sigma(S)$ is almost saturated, there is an element $m \in S$ with $m + S^{sat} \subseteq S$. Consequently, for $e \in \mathcal{R}_\rho(S^\mathrm{sat})$, there exists a character $m' \in m + S^{sat}$ and a non-negative integer $n$ with $m' + ne \in A$. Therefore, $e$ is not a Demazure root of $S$.

    Conversely, assume $\sigma(S)^\vee \cap \rho^\bot$ is almost saturated. Then there is a character ${m \in \rho^\bot \cap S}$ with ${m + S^{sat} \subseteq S}$. Thus, for a ray $\rho'$ different from $\rho$, any $e \in \mathcal{R}_\rho(S^\mathrm{sat})$ with ${\bangle{e, u_{\rho'}} \geq \bangle{m, u_{\rho'}}}$ is a Demazure root of $S$.
\end{proof}

\begin{remark}
    If $\mathcal{R}_\rho(S) \neq \varnothing$ for a ray $\rho \subseteq \sigma(S)$, then $\mathcal{R}_\rho(S)$ is infinite.
\end{remark}


\section{The first case}

In this section we consider surfaces where both one-dimensional torus orbits consist of singular points. This means that the character monoid admits no Demazure root.

Such toric surfaces are rigid. The automorphism group of rigid toric varieties of an arbitrary dimension is described in  \cite[Theorem~3]{BG16}. Here we formulate the explicit desctiption of the automorphism group of an affine toric surface.

Let $T$ be the acting torus on $X$. Fix a basis of the lattice $M$ such that the primitive elements generating the rays of $\sigma(X)^\vee$ are
\[
u_1 = (0, 1) \quad \text{and} \quad u_2 = (d, -e).
\]
Let $t_1, t_2$ be the corresponding coordinates on $T$.

If $e^2 \equiv 1 \pmod d$, then there is a linear transformation of $M$, which permutes $u_1$ and $u_2$. It is given by
\begin{equation*}
        A =
        \begin{pmatrix}
            e & d \\
            \frac{1 - e^2}{d} & -e
        \end{pmatrix}.
    \end{equation*}

\begin{proposition} \label{propCase1}
    Assume the normalization of an affine toric surface $X$ is $X_{d, e}$ and the singular locus of $X$ is the union of two irreducible curves.
Then the automorphism group is
    \begin{equation*}
        \Aut(X) \cong \begin{cases}
                   \ZZ / 2\ZZ \ltimes T, & \text{if} \ e^2 \equiv 1 \pmod d \text{ and } A \text{ preserves } S(X), \\
                   T, & \text{otherwise},
                  \end{cases}
    \end{equation*}
    where the cyclic group $\ZZ / 2\ZZ = \bangle{\tau}$ acts on the torus $T$
via $\tau(t_1, t_2) = \left(t_1^e \, t_2^\frac{1-e^2}{d}, t_1^d \, t_2^{-e}\right)$.
\end{proposition}
\begin{proof}
    Any automorphism of $X$ restricts to an automorphism of the regular locus $X^{\operatorname{reg}} \subseteq X$.
This restriction provides an embedding $\Aut(X) \subseteq \Aut(X^{\operatorname{reg}})$.

    Since the regular locus coincides with the open orbit, its automorphism group is $\Aut(X^{\operatorname{reg}}) \cong \GL_2(\ZZ) \ltimes T$. An element $(B, t) \in \GL_2(\ZZ) \ltimes T$ extends to an automorphism of $X$ if and only if the matrix $B$ preserves the monoid $S(X)$.

This implies that it preserves $\sigma(X)^\vee$, so either $Bu_1 = u_1$, $Bu_2 = u_2$ or $Bu_1 = u_2$, $Bu_2 = u_1$.
In the first case, $B$ is the identity, and in the second case we have $B = A$.

If $A$ preserves $S(X)$, then it defines an automorphism $\tau$ of $X$. Moreover, we have $A^2 = \operatorname{Id}$, so $\bangle{\tau} \cong \ZZ / 2\ZZ$.
\end{proof}

\begin{example}
Consider a product $X = Y_1 \times Y_2$ of non isomorphic singular affine toric curves. The singular locus of $X$ is the union of two irreducible lines and the normalization is $\AA^2$. Thus, $\Aut(X)$ is the acting torus $T$.
\end{example}

\begin{example}
    Consider a non-normal affine toric surface $X$ defined by the monoid shown on Figure \ref{fig:ex1}.
This monoid is generated by three characters
    \[
    m_1 = (2,-1), \quad m_2 = (3,0), \quad m_3 = (0,1),
    \]
    which satisfy the linear relation
    \[
    3m_1 + 3m_3 - 2m_2 = 0.
    \]
    Therefore, the surface $X$ can be embedded as a closed subvariety in $\AA^3$ defined by the equation $x^3y^3=z^2$.
    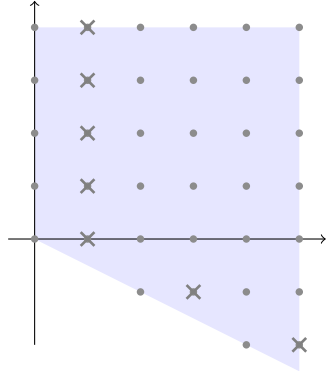
\begin{figure}[htpb]
        \begin{tikzpicture}[scale=0.7]
            \fill[blue!10] (0,0) -- (0,4) -- (5,4) -- (5, -2.5) -- cycle;

            \draw[->] (-0.5, 0) -- (5.5, 0) node[right] {};
            \draw[->] (0, -2) -- (0, 4.5) node[above] {};

            \foreach \x in {0,...,5} {
            \foreach \y in {-3,...,4} {
                \pgfmathtruncatemacro{\twicey}{-2*\y}
                \ifnum \twicey > \x
                \else
                \fill[gray!90] (\x,\y) circle (2pt);
                \fi
            }
            }
            \foreach \x in {0,...,4} {
                \drawcross{1}{\x};
            }
            \drawcross{3}{-1}
            \drawcross{5}{-2}
        \end{tikzpicture}
        \caption{The character monoid of the surface $X = \{x^3y^3=z^2\} \subset \AA^3$}
    \label{fig:ex1}
    \end{figure}

    This surface has no Demazure root, and the normalization of $X$ is $X_{2, 1}$. Consequently, the automorphism group of $X$ is $\Aut(X) = \bangle{\tau} \ltimes T$ with the involutive automorphism
    $$\tau(x, y, z) = (y, x, z).$$
\end{example}

\section{The second case}
\label{nr}

Now we consider the case when the singular locus of $X$ is an irreducible curve. Such toric surfaces are known as almost rigid.

The automorphism group of an affine almost rigid toric variety of an arbitrary dimension is described in \cite[Theorem 5]{BG16}. Now we formulate the description of the automorphism group of an almost rigid affine toric surface.

Recall that a Demazure root $e \in M$ of $S$, associated with a ray $\rho \subseteq \sigma(X)$, induces a \mbox{$\GG_a$-action} on $X$. We denote by $\varphi_e$ the corresponding $\GG_a$-subgroup in $\Aut(X)$. For an arbitrary parameter $\alpha \in k$, the induced automorphism $\varphi_e^*(\alpha)$ of the coordinate algebra acts on characters via
\[
\varphi_e^*(\alpha)(\chi^m) = \sum_{i=0}^{\langle m, u_\rho\rangle} \binom{\langle m, u_\rho\rangle}{i} \alpha^i \chi^{m + ie}.
\]

Let $\ti{X}$ be the normalization of $X$. Since $\Aut(X)$ is embedded into  $\Aut(\ti{X})$, any automorphism from the connected component of unity of $\Aut(X)$ can be viewed as a product of torus elements and root automorphisms of $\ti{X}$.

We call a ray $\rho \subseteq \sigma(X)$ singular if the associated one-dimensional torus orbit $O(\rho)$ consists of singular points; otherwise, a ray $\rho$ is smooth.

\begin{lemma} \label{oneRayLemma}
    Suppose that a surface $X$ contains one singular one-dimensional torus orbit $O(\rho)$. Then every automorphism of $X$ factors as a product of a torus element and root automorphisms of $\ti{X}$ associated with the singular ray $\rho$.
\end{lemma}

\begin{proof}
    Let $\rho'$ denote the smooth ray and $\rho$ the singular ray. We can embed the group $\Aut(X)$ into $\Aut(X^{\operatorname{reg}})$, where $X^{\operatorname{reg}} = X \setminus \overline{O(\rho)}$ is isomorphic to $\AA^1 \times T^1$.

    Let $m \in S(X)^{sat} = S(\ti{X})$ be the character generating the sublattice $\rho^\bot \cap M = M(\rho)$, and $e_0$ be a Demazure root of $S(\ti{X})$ that minimizes the pairing $\langle e_0, u_{\rho'}\rangle$.

    The characters $m$ and $m-e_0$ form a basis of the lattice $M$. We will use these characters to define coordinates on $X^{\text{reg}}$. First, we show that $m - e_0 \in S(\ti{X})$. This is equivalent to $\langle m - e_0, u_{\rho'}\rangle \geq 0$. Assume for a contradiction that $\langle m - e_0, u_{\rho'}\rangle < 0$. Setting $e = e_0 - m$, we obtain
    \[
    \langle e, u_\rho\rangle = \langle e_0, u_\rho\rangle = -1 \quad \text{and} \quad \langle e, u_{\rho'}\rangle > 0.
    \]
    Thus, $e$ is a Demazure root of $S(\ti{X})$ associated with the ray $\rho$. Since $m \in S(\ti{X})$ and $m \notin (\rho')^\bot$, we have $\langle m, u_{\rho'}\rangle > 0$. This implies $\langle e, u_{\rho'}\rangle = \langle e_0, u_{\rho'}\rangle - \langle m, u_{\rho'}\rangle < \langle e_0, u_{\rho'}\rangle$, which contradicts the definition of $e_0$.

    Therefore, functions $t = \chi^m$ and $x = \chi^{m - e_0}$ are regular on $\ti{X}$. Moreover, these functions are coordinates on $T^1$ and $\mathbb{A}^1$, respectively.

    In terms of these coordinates, any automorphism of $X^{\text{reg}}$ takes the form
    \[
    (t, x) \mapsto (\alpha t^{\pm1}, \beta t^ax + f(t)),
    \]
    where $\alpha, \beta \in k^\times$, $a \in \ZZ$, and $f \in k[t, t^{-1}]$. The inverse of such an automorphism is given by
    \[
    (t, x) \mapsto (\alpha^{\mp1}t^{\pm1}, \beta^{-1}\alpha^at^{-a}(x - f(t))).
    \]

    We now show when such an automorphism lies in the subgroup ${\Aut(X)\subseteq\Aut(X^{\operatorname{reg}})}$. If the automorphism extends to $X$, the functions $\alpha t^{\pm 1}$, $\alpha^{\mp1}t^{\pm1}$, $\beta t^ax + f(t)$, and ${\beta^{-1}\alpha^at^{-a}(x - f(t))}$ are regular on $X$. In particular, they lie in $k[S^{sat}]$. This restricts the automorphism to the form
    \[
    (t, x) \mapsto (\alpha t, \beta x + f(t)),
    \]
    where $\alpha, \beta \in k^\times$ and $f \in k[t]$.

    Any such automorphism of $X^{\operatorname{reg}}$ factors as a product of a torus element and automorphisms of the form
    \[
    \psi \colon (t, x) \mapsto (t, x + \gamma t^s),
    \]
    where $\gamma \in k$ and $s \in \ZZ_{\geq 0}$.

    For $s \geq 1$, we have a root $e = e_0 + (s - 1)m \in \mathcal{R}_\rho(\ti{X})$. The corresponding root automorphism $\varphi_e(\gamma)$ acts on the coordinate functions $t$ and $x$ via
    \[
    t \mapsto t, \quad x \mapsto x + \gamma t^s.
    \]
    Therefore, for $s \geq 1$, $\psi = \varphi_e(\gamma)$.

    If $s = 0$ and $\gamma \neq 0$, we claim that $\psi$ does not extend to an automorphism of $\widetilde{X}$. We show this by considering two cases:
    \begin{enumerate}
     \item Suppose that $\widetilde{X}$ is singular. Assume for a contradiction that $\psi$ extends to an automorphism of $X$, and consequently to $\widetilde{X}$. Since the unique closed torus orbit in $\widetilde{X}$ is the only singular point, it is fixed by any automorphism.

    Consider the closed subvariety $Y = V(t,x) \subset \widetilde{X}$. The closure $\overline{O(\rho)}$ is defined by the equation $t = 0$, which implies $Y \subseteq \overline{O(\rho)}$. Because $Y$ is a closed torus-invariant subset, it contains the closed torus orbit $O(\sigma(\widetilde{X}))$.

    For any point $p \in Y$, we have
    \[
    x(\psi(p)) = \psi^*(x)(p) = x(p) + \gamma = \gamma \neq 0.
    \]
    Thus, $Y \cap \psi(Y) = \varnothing$. This contradicts the fact that $\psi$ fixes the singular point $O(\sigma(\widetilde{X})) \in Y$.

     \item Assume $\ti{X} \cong \AA^2$. Fix a basis in $M$ such that $m = (0, 1)$, $e_0 = (-1, 0)$, and the primitive character of $S(\ti{X}) \cap (\rho')^\bot$ is $m' = (1, 0)$.

    The automorphism $\psi$ acts on the function $\chi^{m'} = \chi^{-e_0} = \chi^{m - e_0}\chi^{-m} = xt^{-1}$ via
     \[
     \chi^{m'} \mapsto \psi^*(x)\psi^*(t^{-1}) = (x + \gamma)t^{-1} = \chi^{m'} + \gamma \chi^{-m}.
     \]
     Since the character $\chi^{-m}$ is not regular on $\ti{X}$, $\psi$ is not an automorphism of $\ti{X}$.
    \end{enumerate}
\end{proof}

 Let $\mathcal{U}$ denote the subgroup of $\operatorname{Aut}(X)$ generated by all root automorphisms. In our setting every Demazure root of $X$ is associated with the same singular ray, which implies that all root subgroups commute. Thus, the group $\mathcal{U}$ is isomorphic to the additive group of a countably dimensional vector space.

 \begin{proposition} \label{rootAutProp}
    If an affine toric surface $X$ has an irreducible one-dimensional singular locus, its automorphism group is generated by the acting torus $T$ and root subgroups.
    Moreover, $\Aut(X)$ is isomorphic to the semidirect product $\mathcal{U} \rtimes T$.
\end{proposition}

\begin{proof}
    Recall that any irreducible one-dimensional singular locus of a toric surface coincides with the closure of a one-dimensional torus orbit $\overline{O(\rho)}$ for some ray $\rho \subseteq \sigma(X)$.

    By Lemma \ref{oneRayLemma}, every automorphism of $X$ factors into a product of a torus element and root automorphisms of $\ti{X}$ associated with the singular ray $\rho$. It suffices to show that this factorization consists only of root automorphisms of $X$ itself.

    Assume for a contradiction that there exists an automorphism $\psi \in \operatorname{Aut}(X)$ whose factorization
    \[
    \psi = \varphi_{e_1}(\alpha_1)\cdots \varphi_{e_n}(\alpha_n)
    \]
    has root automorphisms of $\widetilde{X}$ that do not preserve $X$. Here for each $j = 1, \dots, n$ ${e_j \in \mathcal{R}_\rho(\ti{X})}$.

    Since all roots $e_j$ correspond to the common ray $\rho$, the associated root subgroups commute. Without loss of generality, we may assume that roots $e_j$ are distinct and  $\varphi_{e_n}(\alpha_n) \notin \operatorname{Aut}(X)$

    Since $e_n$ is not a Demazure root of $S(X)$, by definition there exists a character $m \in S(X)$ with $\langle m, u_{\rho}\rangle > 0$ and $m + e_n \notin S(X)$.
    Now consider a function $\psi^*(\chi^m)$, which can be expressed as a linear combination of characters. We show that this linear combination contains a non-regular character $\chi^{m+e_n}$.

    Every character in the linear combination takes the form
    \[
    \chi^{m + i_1 e_1 + \dots + i_n e_n},
    \]
    where $i_j \in \{0, 1, \dots, \langle m, u_\rho \rangle\}$ for each $j = 1, \dots, n$.
    Suppose that for some indices $i_1, \dots, i_n$ we have
    \[
    m + i_1e_1 + \dots + i_ne_n = m + e_n,
    \]
    \[
    i_1e_1 + \dots + i_{n-1}e_{n-1} + (i_n - 1)e_n = 0.
    \]
    Pairing this expression with $u_\rho$, we get
    \[
    \langle i_1e_1 + \dots + i_{n-1}e_{n-1} + (i_n - 1)e_n, u_\rho\rangle = 1 - i_1 - \dots - i_n = 0.
    \]

    Since $i_j$ are non-negative integers, all the indices $i_j$ are zero, except for some index $l$ for which $i_l = 1$. This implies $e_l = e_n$. Since the roots are distinct, we have $l = n$, so
    \[
    i_1 = \dots = i_{n-1} = 0 \quad \text{and} \quad i_n = 1.
    \]
    Thus, the character $\chi^{m+e_n}$ appears in $\psi^*(\chi^m)$ with a non-zero coefficient $\langle m, u_\rho \rangle \alpha_n \neq 0$. This implies $\psi^*(\chi^m) \notin k[X]$, which contradicts the assumption that $\psi \in \operatorname{Aut}(X)$. Thus, $\Aut(X)$ is generated by $T$ and root subgroups.

    It is clear that $T \cap \mathcal{U} = \{\operatorname{id}_X\}$ and that $T$ normalizes $\mathcal{U}$ since it normalizes all the root subgroups. Therefore, $\operatorname{Aut}(X) = \mathcal{U} \rtimes T$.

\end{proof}

\begin{corollary}
    The automorphism group $\Aut(X)$ of a non-normal affine toric suface $X$ with an irreducible one-dimensional singular locus is connected.
\end{corollary}

\begin{proof}
    We will show that every automorphism $\varphi \in \Aut(X)$ lies in a connected algebraic family of automorphisms containing the identity.
    By Proposition \ref{rootAutProp}, the automorphism $\varphi$ can be written as
    \[
    \varphi = t\cdot \varphi_{e_1}(\alpha_1)\cdots\varphi_{e_s}(\alpha_s),
    \]
    where $t \in T$ is a torus element, the elements $e_1, \dots, e_s$ are Demazure roots of $X$ associated with the singular ray $\rho$, and $\alpha_1, \dots, \alpha_s \in k$.

    We define an algebraic family of automorphisms parameterized by $T \times \AA^s$ via
    \[
    (t, (\alpha_1, \dots, \alpha_s)) \mapsto t\cdot \varphi_{e_1}(\alpha_1)\cdots\varphi_{e_s}(\alpha_s).
    \]
    This family is connected and contains both $\operatorname{id}_X$ (when $t = 1$ and $\alpha_i = 0$) and $\varphi$. Hence, $\Aut(X)$ is connected.
\end{proof}

\begin{remark}
    The automorphism group of a non-normal toric surface can be connected even if the automorphism group of its normalization is not.
For example, consider any almost rigid toric surface with normalization $X_{3,1}$; the automorphism group of $X_{3,1}$ is not connected.
The question of connectedness of the automorphism group for normal affine toric varieties is considered in \cite{Kikt24}.
\end{remark}

\begin{corollary}
    Let $Y$ be a singular affine toric curve.
The automorphism group $\Aut(Y \times \AA^1)$ is generated by the acting torus and root subgroups.
\end{corollary}
\begin{proof}
    The monoid $S(Y \times \mathbb{A}^1)$ satisfies the condition (2) of Lemma \ref{rhoSing}, so the singular locus of $Y \times \mathbb{A}^1$ is an irreducible one-dimensional curve. Thus, by Proposition \ref{rootAutProp}, $\Aut(Y \times \mathbb{A}^1)$ is generated by torus elements and root automorphisms.
\end{proof}

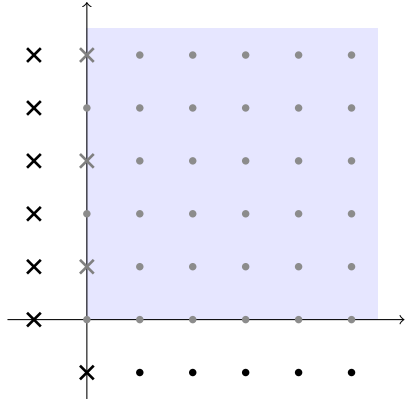
\begin{figure}[htpb]
    \begin{tikzpicture}[scale=0.7, baseline=(current bounding box.center)]

    \fill[blue!10] (0,0) -- (0,5.5) -- (5.5,5.5) -- (5.5, 0) -- cycle;

    \draw[->] (-1.5, 0) -- (6, 0) node[right] {};
    \draw[->] (0, -1.5) -- (0, 6) node[above] {};

    \foreach \x in {0,...,5} {
    \foreach \y in {0,...,5} {
        \ifnum \x=0
            \ifodd\y
            \else
                \fill[gray!90] (\x,\y) circle (2pt);
            \fi
        \else
            \fill[gray!90] (\x,\y) circle (2pt);
        \fi
    }
    }

    \foreach \x in {1,3,5} {
        \drawcross{0}{\x}
    }

    \foreach \x in {0,...,5} {
        \drawblackcross{-1}{\x}
    }
    \foreach \x in {0,...,5} {
        \ifnum \x=0
        \else
        \fill[black] (\x,-1) circle (2pt);
        \fi
    }
    \drawblackcross{0}{-1}

    \end{tikzpicture}
    \caption{\small The character monoid and Demazure roots of the Whitney umbrella}
    \label{fig:ex2_whitney}
\end{figure}

\begin{example}
    The surface $X \subset \AA^3$ defined by $x^2y = z^2$, known as the Whitney umbrella, is an affine toric variety with the character monoid shown in Figure \ref{fig:ex2_whitney}. The singular locus is an irreducible curve $\{x = z = 0\}$ and every Demazure root of this surface is associated with a single ray of $\sigma(X)$.

    If an automorphism $\varphi \in \Aut(X^{\operatorname{reg}})$ of the open subset $X^{\operatorname{reg}} = X \setminus \{x = z = 0\}$ extends to an automorphism of $X$, the corresponding algebra automorphism $\varphi^*$ acts on the coordinate functions $x$ and $w = \chi^{(0, 1)} \in k[X^{\operatorname{reg}}]$ by
    \[
    \varphi^*(x) = \alpha x, \quad \varphi^*(w) = \beta w + f(x),
    \]
    for some $\alpha, \beta \in k^\times$ and $f \in k[x]$.
    Since $z = w^2$ and $y = xw$, we evaluate $\varphi^*$ on $y$ and $z$ as follows
    \begin{align*}
    \varphi^*(y) &= \alpha x (\beta w + f(x)) = \alpha\beta y + \alpha x f(x), \\
    \varphi^*(z) &= (\beta w + f(x))^2 = \beta^2 z + 2\beta wf(x) + f(x)^2.
    \end{align*}
    Therefore, $\varphi^*(y), \varphi^*(z) \in k[X]$ if and only if $wf(x) \in k[X]$, which is equivalent to $f \in xk[x]$.
    Hence, any automorphism $\varphi$ of $X$ takes the form
    \[
    (x, y, z) \mapsto (\alpha x, \alpha\beta y + \alpha x^2g(x), \beta^2z + 2\beta yg(x) + x^2g(x)^2),
    \]
    for some $\alpha, \beta \in k^\times$ and $g \in k[x]$.
    Any such automorphism factors as a product of a torus element and root automorphisms of the form
    \[
    (x, y, z) \mapsto (x, y + \alpha x^{s+2}, z + 2\alpha yx^s + \alpha^2x^{2s+2}),
    \]
    where $s \geq 0$ and $\alpha \in k$.
    These are root automorphisms associated with the Demazure roots $e_s = (s+1, -1)$.

    The Whitney umbrella is an example of a Danielewski surface. Automorphisms of such varieties are described in \cite[Theorem~7.11]{Gaif21}.
\end{example}

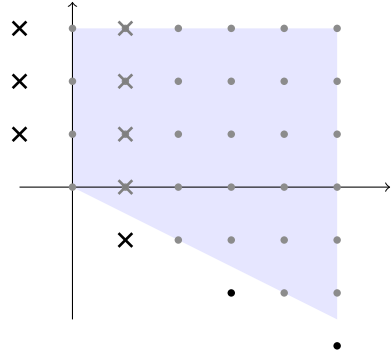
\begin{figure}[htpb]
    \centering
    \begin{tikzpicture}[scale=0.7]
        \fill[blue!10] (0,0) -- (0,3) -- (5,3) -- (5, -2.5) -- cycle;

        \draw[->] (-1, 0) -- (6, 0) node[right] {};
        \draw[->] (0, -2.5) -- (0, 3.5) node[above] {};

        \foreach \x in {0,...,5} {
            \foreach \y in {-3,...,3} {
                \pgfmathtruncatemacro{\twicey}{-2*\y}
                \ifnum \twicey > \x
                \else
                    \fill[gray!90] (\x,\y) circle (2pt);
                \fi
            }
        }

        \foreach \x in {0,...,3} {
            \drawcross{1}{\x};
        }

        \foreach \x in {1,...,3} {
            \drawblackcross{-1}{\x}
        }
        \drawblackcross{1}{-1}
        \fill[black] (3,-2) circle (2pt);
        \fill[black] (5,-3) circle (2pt);
    \end{tikzpicture}
    \caption{The character monoid and Demazure roots of the affine surface $X = \{x^3y=z^2\} \subset \AA^3$}
    \label{fig:ex3}
\end{figure}

\begin{example}
    Consider the surface $X = \{x^3y = z^2\} \subset \AA^3$, whose normalization is $\widetilde{X} = X_{2,1}$. The singular locus of $X$ is $\{x = z = 0\}$. If an automorphism $\varphi \in \operatorname{Aut}(X^{\operatorname{reg}})$ extends to~$X$, the corresponding algebra automorphism $\varphi^*$ acts on the coordinate functions $x$ and $w = \chi^{(1,0)}$ of $X^{\operatorname{reg}}$ by
    \[
    \varphi^*(x) = \alpha x, \quad \varphi^*(w) = \beta w + f(x),
    \]
    for some $\alpha, \beta \in k^\times$ and $f \in k[x]$. Since $z = xw$ and $y = w^2x^{-1}$, we evaluate $\varphi^*$ on $y$ and $z$ as follows:
    \begin{align*}
    \varphi^*(y) &= \alpha^{-1}x^{-1}(\beta w + f(x))^2 = \alpha^{-1}(\beta^2y + 2\beta x^{-1}wf(x) + x^{-1}f(x)^2), \\
    \varphi^*(z) &= \alpha x(\beta w + f(x)) = \alpha\beta z + \alpha xf(x).
    \end{align*}
    Therefore, $\varphi^*(y), \varphi^*(z) \in k[X]$ if and only if $f \in x^2k[x]$, that is, $f = x^2g$ for some $g \in k[x]$. Consequently, any automorphism of $X$ takes the form
    \[
    (x, y, z) \mapsto (\alpha x, \alpha^{-1}\beta^2 y + 2\alpha^{-1}\beta zg(x) + \alpha^{-1}x^3 g(x)^2, \alpha\beta z + \alpha x^3 g(x)),
    \]
    which factors as a product of the torus element $t = (\alpha, \beta)$ and root automorphisms of the form
    \[
    (x, y, z) \mapsto (x, y + 2\gamma zx^s + \gamma^2 x^{2s+3}, z + \gamma x^s),
    \]
    for some $\gamma \in k$. These are root automorphisms associated with the Demazure roots ${e_s = (2s+3, -s-2)}$.
\end{example}


\section{The third case}
\label{cc}

Now we consider the case when the singular locus of $X$ is the $T$-fixed point. In this case the number of holes in $S(X)$ is finite. This implies that the number of Demazure roots associated with each of two rays of the cone $\sigma(X)$ is infinite.

Denote by $\Aut^*(X)$ the subgroup of $\Aut(X)$ generated by the acting torus $T$ and root subgroups. In the previous cases the subgroup $\Aut^*(X)$ coincides with the connected component $\Aut(X)^0$ of the automorphism group $\Aut(X)$. Let us give an example where $\Aut^*(X)$ is a proper subgroup of $\Aut(X)^0$.

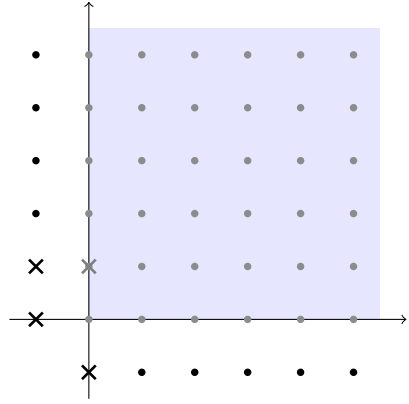
\begin{figure}[htbp]
    \begin{tikzpicture}[scale=0.7, baseline=(current bounding box.center)]

    \fill[blue!10] (0,0) -- (0,5.5) -- (5.5,5.5) -- (5.5, 0) -- cycle;

    \draw[->] (-1.5, 0) -- (6, 0) node[right] {};
    \draw[->] (0, -1.5) -- (0, 6) node[above] {};

    \foreach \x in {0,...,5} {
    \foreach \y in {0,...,5} {
        \ifnum \x=0
            \if \y in {1,...,5}
            \else
                \fill[gray!90] (\x,\y) circle (2pt);
            \fi
        \else
            \fill[gray!90] (\x,\y) circle (2pt);
        \fi
    }
    }
    \drawcross{0}{1}

    \drawblackcross{-1}{0}
    \drawblackcross{-1}{1}
    \foreach \x in {2,...,5} {
        \fill[black] (-1,\x) circle (2pt);
    }
    \foreach \x in {0,...,5} {
        \ifnum \x=0
        \else
        \fill[black] (\x,-1) circle (2pt);
        \fi
    }
    \drawblackcross{0}{-1}

    \end{tikzpicture}
    \caption{The character monoid and Demazure roots of the affine surface $X = \{z^3=w^2, zx^2=v^2, xw = vz, vw = xz^2\} \subset \AA^4$}
    \label{fig:ex4}
\end{figure}

Let $X$ be a non-normal toric surface whose normalization is $\AA^2$ and the only hole is $(0, 1)$. The character monoid $S(X)$ is generated by
$$u_x = (1, 0), \quad u_v = (1, 1), \quad u_z = (0, 2), \quad u_w = (0, 3).$$
Letting $x =\chi^{u_x}, v = \chi^{u_v}, z = \chi^{u_z}, w = \chi^{u_w}$, we obtain that $k[X]$ is the quotient of $k[x, v, z, w]$ modulo the ideal generated by
$$z^3 - w^2, \quad zx^2 - v^2, \quad xw - vz, \quad vw - xz^2.$$

\begin{theorem} \label{thmCase3}
    There exists an irreducible algebraic family of automorphisms $\varphi(\alpha, \beta)$ of $X$ with $\alpha, \beta \in k$ which is not contained in $\Aut^*(X)$. In particular, $\Aut^*(X)$ is a proper subgroup of $\Aut(X)^0$.
\end{theorem}
\begin{proof}
The normalization of $X$ is the affine plane $\AA^2$, so $\Aut(X) \subseteq \Aut(\AA^2)$. One easily checks that the set of Demazure roots of $S(X)$ is the same as for $\AA^2$ except for three roots $(0, -1), (-1, 0)$ and $(-1, 1)$.
Consider an algebraic family of automorphisms
\begin{equation*}
\varphi(\alpha, \beta) = \varphi_{e_1}(-\alpha)\varphi_{e_2}(\beta)\varphi_{e_1}(\alpha) \in \Aut(\AA^2),
\end{equation*}
where $\alpha, \beta \in k$ and $e_1 = (-1, 1)$, $e_2 = (2, -1)$ are Demazure roots of $\ti{X} \cong \AA^2$.

We show that $\varphi^*(\alpha, \beta)\left(k[X]\right) \subseteq k[X]$. Both root automorphisms $\varphi_{e_1}$ and $\varphi_{e_2}$ map the functions $v, z$ and $w$ to $k[X]$, while the action of $\varphi(\alpha, \beta)$ on $x$ is
    $$\varphi^*(\alpha, \beta)(x) = \varphi^*_{e_1}(\alpha)\varphi^*_{e_2}(\beta)\varphi^*_{e_1}(-\alpha)(x) = \varphi^*_{e_1}(\alpha)\varphi^*_{e_2}(\beta)(x - \alpha y) = $$
    $$= \varphi^*_{e_1}(\alpha)(x - \alpha y - \alpha\beta x^2) = x - \alpha\beta x^2 - 2\alpha^2\beta v - \alpha^3\beta z \in k[X],$$
    where $y = \chi^{(0, 1)} \in k[\AA^2] \setminus k[X]$. This shows that $\varphi(\alpha, \beta) \in \Aut(X)$ for any $\alpha, \beta \in k$.
Also $\varphi(0, 0) = \operatorname{id}_X$, so the family $\varphi(\alpha, \beta)$ is contained in $\Aut(X)^0$.

Now we show that the family $\varphi(\alpha, \beta)$ is not contained in $\Aut^*(X)$. By Jung-van der Kulk theorem, the automorphism group $\Aut(\mathbb{A}^2)$ has the structure of an amalgam
\[
    \Aut(\mathbb{A}^2) = \mathrm{JONQ}^-(\mathbb{A}^2) *_{\operatorname{Aff}^-(\AA^2)} \operatorname{Aff}(\AA^2).
\]

Let us choose the system of coset representatives of $\mathrm{Aff}^-(\AA^2)$ in $\mathrm{JONQ}^-(\AA^2)$, so it consists of the transformations
$$(x, y) \mapsto (x, y + f(x))$$
for $f \in x^2k[x]$. The system of coset representatives of $\mathrm{Aff}^-(\AA^2)$ in $\mathrm{Aff}(\AA^2)$ can be chosen as follows. In the coordinates $(x, y)$ all the transformations in $\mathrm{Aff}(\AA^2)$ take the form
$$(x, y) \mapsto (a_{11}x + a_{12}y + b_1, a_{21}x + a_{22}y + b_2).$$
If ${a_{11} \neq 0}$ then the corresponding coset is represented by automorphism ${(x, y) \mapsto (x + \frac{a_{12}}{a_{11}}y, y)}$. Otherwise, it is represented by the transposition ${\nu_0: (x, y) \mapsto (y, x)}$.

Then any transformation $\psi \in \Aut(\AA^2)$ can be uniquely decomposed as
$$\psi = c \cdot \mu_1 \cdot \dots \mu_s,$$
where $c \in \mathrm{Aff}^-(\AA^2)$ and for $j = 1, \dots, s$ automorphisms $\mu_j$ are the chosen representatives either in $\mathrm{JONQ}^-(\AA^2)$ or $\mathrm{Aff}(\AA^2)$.

Conjugation by $\nu_0$ swaps the upper and lower-triangular root subgroups. Moreover, any root automorphism is either chosen coset representative in $\mathrm{JONQ}^-(\AA^2)$ or is conjugate to one by $\nu_0$.

Thus, distinct decompositions of $\phi(\alpha, \beta)$ into a product of root automorphisms give distinct decompositions into a product of coset representatives. Therefore, the given decomposition of $\phi(\alpha, \beta)$ into root subgroups is unique in $\Aut(\AA^2)$. It follows that $\varphi(\alpha, \beta)$ is not a product of root automorphisms of $X$.

\end{proof}

For some non-normal affine toric surfaces $X$ the group $\Aut^*(X)$ can coincide with $\Aut(X)^0$. For example, this is the case when $\mathcal{R}(X) = \mathcal{R}(\ti{X})$.
In \cite[Proposition~4.3]{DL23} it is shown that this happens, in particular, when $\ti{X}$ is a singular toric surface and the set of holes coincides with the set of indecomposible elements of the semigroup $S(\ti{X})$. But in general the structure of the group $\Aut(X)$ in the third case is not clear and deserves further investigation.


\end{document}